\documentclass[12pt]{amsart}

\usepackage[english]{babel}

\usepackage[utf8]{inputenc}

\usepackage{microtype,amsmath,amsthm,amssymb}
\usepackage{parskip}
\usepackage{geometry}
\usepackage[hidelinks]{hyperref}
\author{Paul Pollack} 
\address{Department of Mathematics \\ University of Georgia \\ Athens, GA 30602}
\email{pollack@uga.edu}
\thanks{P.P. is supported by NSF award DMS-2001581.}
 
\subjclass{Primary 11N37; Secondary 11A25, 11N64}

\author{Akash Singha Roy}
\address{ESIC Staff Quarters No.: D2\\ 143 Sterling Road, Nungambakkam\\Chennai 600034\\ Tamil Nadu, India.}
\email{akash01s.roy@gmail.com}

\renewcommand\phi\varphi

\renewcommand{\pod}[1]{\allowbreak\mathchoice
  {\if@display \mkern 18mu\else \mkern 8mu\fi (#1)}
  {\if@display \mkern 18mu\else \mkern 8mu\fi (#1)}
  {\mkern4mu(#1)}
  {\mkern4mu(#1)}
}
\usepackage{graphicx}
\DeclareMathAlphabet{\curly}{U}{rsfs}{m}{n}

\newtheorem{thm}{Theorem}[section]

\newtheorem{prop}[thm]{Proposition}
\newtheorem{lem}[thm]{Lemma}
\newtheorem{conj}[thm]{Conjecture}

\theoremstyle{remark}
\newtheorem{rmk}[thm]{Remark}

\title{Powerfree sums of proper divisors}

\begin{document}
\begin{abstract} Let $s(n):= \sum_{d\mid n,~d<n} d$ denote the sum of the proper divisors of $n$. It is natural to conjecture that for each integer $k\ge 2$, the equivalence
\[ \text{$n$ is $k$th powerfree} \Longleftrightarrow \text{$s(n)$ is $k$th powerfree} \]
holds almost always (meaning, on a set of asymptotic density $1$). We prove this for $k\ge 4$.
\end{abstract}

\maketitle

\section{Introduction}
A 19th century
theorem of Gegenbauer asserts that for each fixed $k$, the set of positive integers not divisible by the $k$th power of an integer larger than $1$ has asymptotic density $\zeta(k)^{-1}$, where $\zeta(s)$ is the familiar Riemann zeta function. Call such an integer \textsf{$k$th-power-free}, or \textsf{$k$-free} for short. 
In this note we investigate the frequency with which the sum-of-proper-divisors function $s(n):=\sum_{d\mid n,~d<n} d$ assumes $k$-free values. As we proceed to explain, there is a natural guess to make here, formulated below as Conjecture \ref{conj:main}.

Fix $k\ge 2$. If $n$ is not $k$-free, then $p^k \mid n$ for some prime $p$. Moreover, if $y=y(x)$ is any function tending to infinity, the the upper density of $n$ divisible by $p^k$ for some $p > y^{1/k}$ is at most $\sum_{p > y^{1/k}}p^{-k} = o(1)$. Hence, almost always a non $k$-free number $n$ is divisible by $p^k$ for some $p^k \le y$.  To be precise, when we say a statement about positive integers $n$ holds \textsf{almost always}, we mean that it holds for all $n\le x$ with $o(x)$ exceptions, as $x\to\infty$. (Importantly, we allow the statement itself to involve the growing upper bound $x$.)  

It was noticed by Alaoglu and Erd\H{o}s \cite{AE44} that whenever $y=y(x)$ tends to infinity with $x$ slowly enough, $\sigma(n)$ is divisible by all of the integers in $[1,y]$ almost always. (We give a proof below with $y:=(\log\log{x})^{1-\epsilon}$; see Lemma \ref{lem:easylem}.) Hence, almost always $n$ and $s(n)=\sigma(n)-n$ share the same set of divisors up to $y$. Putting this together with the observations of the last paragraph, we see that if $n$ is not $k$-free, then  $s(n)$ is not $k$-free, almost always. The same reasoning shows that if $n$ is $k$-free, then $s(n)$ is not divisible by $p^k$ for any $p \le y^{1/k}$, almost always. Thus, if it could be shown that almost always $s(n)$ is not divisible by $p^k$ for any prime $p > y^{1/k}$, then we would have established the following conjecture.

\begin{conj}\label{conj:main} Fix $k\ge 2$. On a set of integers $n$ of asymptotic density $1$, 
\[ \text{$n$ is $k$-free} \Longleftrightarrow \text{$s(n)$ is $k$-free}. \]
\end{conj}

The case $k=2$ of Conjecture \ref{conj:main} is alluded to by Luca and Pomerance in \cite{LP15} (see Lemma 2.2 there and the discussion following). Their arguments show that $s(n)$ is squarefree on a set of positive lower density (in fact, of lower density at least $\zeta(2)^{-1} \log{2}$). Conjecture \ref{conj:main}, for every $k \ge 2$, would follow from a very general conjecture of Erd\H{o}s--Granville--Pomerance--Spiro \cite{EGPS90} that the $s$-preimage of a density zero set also has density zero; see Remark \ref{rmk:EGPS} below.

Our result is as follows.

\begin{thm} \label{thm:main} Conjecture \ref{conj:main} holds for each $k\ge 4$.
\end{thm}

To prove Conjecture \ref{conj:main} for a given $k$, it is enough (by the above discussion) to show that almost always $s(n)$ is not divisible by $p^k$ for any $p^k > (\log\log{x})^{0.9}$. The range $p \le x^{o(1)}$ can be treated quickly using familiar arguments (versions of which appear, e.g., in \cite{pollack14}). The main innovation in our argument --- and the source of the restriction to $k\ge 4$ --- is the handling of larger $p$ using a theorem of Wirsing \cite{wirsing59} that bounds the ``popularity'' of values of the function $\sigma(n)/n$.

The reader interested in other work on powerfree values of arithmetic functions may consult \cite{pappalardi03,PSS03,BL05,BP06} as well as the survey \cite{pappalardi05}.
\subsection*{Notation and conventions} We reserve the letters $p, q, P$, with or without subscripts, for primes and we write $\log_k$ for the $k$th iterate of the natural logarithm. We write $P^{+}(n)$ and $P^{-}(n)$ for the largest and smallest prime factors of $n$, with the conventions that $P^{+}(1)=1$ and $P^{-}(1)=\infty$. We adopt the Landau--Bachmann--Vinogradov notation from asymptotic analysis, with all implied constants being absolute unless specified otherwise. 

\section{Preliminaries}
The following lemma is due to Pomerance (see \cite[Theorem 2]{pomerance77}).

\begin{lem}\label{lem:pomerance} Let $a, k$ be integers with $\gcd(a,k)=1$ and $k > 0$. Let $x\ge 3$. The number of $n\le x$ for which there does not exist a prime $p\equiv a\pmod{k}$ for which $p\parallel n$ is $O(x (\log{x})^{-1/\phi(k)})$.
\end{lem} 

The following lemma justifies the claim in the introduction that $\sigma(n)$ is usually divisible by all small primes. It is well-known but, for lack of a suitable reference, we include the short proof.

\begin{lem}\label{lem:easylem} Fix $\epsilon > 0$. Almost always, the number $\sigma(n)$ is divisible by every positive integer $d\le (\log_2 x)^{1-\epsilon}$. 
\end{lem}

\begin{proof} Notice that $d\mid \sigma(n)$ whenever there is a prime $p\equiv -1\pmod{d}$ such that $p \parallel n$. For each $d\le (\log_2 x)^{1-\epsilon}$, the number of $n\le x$ for which there is no such $p$ is $O(x \exp(-(\log_2 x)^{\epsilon}))$, by Lemma \ref{lem:pomerance}. Now sum on $d\le (\log_2 x)^{1-\epsilon}$.
\end{proof}


Our next lemma bounds the number of $n \leq x$ for which $n$ and $\sigma(n)$ possess a large common prime divisor.

\begin{lem}\label{gcd(n,sigma(n)) log log x smooth} Almost always, the greatest common divisor of $n$ and $\sigma(n)$ has no prime divisor exceeding $\log \log x$.
\end{lem}

With more effort, it could be shown that $\gcd(n,\sigma(n))$ is almost always the largest divisor of $n$ supported on primes not exceeding $\log\log{x}$. Compare with Theorem 8 in \cite{ELP08}, which is the corresponding assertion with $\sigma(n)$ replaced by $\phi(n)$.

\begin{proof} Put $y:=\log_2 x$. We start by removing those $n\le x$ with squarefull part exceeding $\frac{1}{2}y$. The number of these $n$ is $O(xy^{-1/2})$, which is $o(x)$ and hence negligible. 

Suppose that $n$ survives and there is a prime $p > y$ dividing $n$ and $\sigma(n)$. Since $p \mid \sigma(n)$, we can choose a prime power $q^e \parallel n$ for which $p \mid \sigma(q^e)$. Then $y < p \leq \sigma(q^e) < 2q^e$, forcing $e=1$. Hence, $p\mid \sigma(q)=q+1$ and $q \equiv -1 \pmod p$. Since $pq\mid n$, we deduce that the number of $n$ belonging to this case is at most
\[ \sum_{p > y} \sum_{\substack{q \equiv -1\pmod{p} \\ q\le x}} \frac{x}{pq} \ll x\sum_{p > y} \frac{1}{p} \sum_{\substack{q \le x \\ q\equiv -1\pmod{p}}} \frac{1}{q} \ll x\log_2 x\sum_{p > y} \frac{1}{p^2} \ll \frac{x\log_ 2 x}{y\log y} = \frac{x}{\log_3 x}, \]
which is again $o(x)$. Here the sum on $q$ has been estimated by the Brun--Titchmarsh inequality (see, e.g., Theorem 416 on p.\ 83 of \cite{tenenbaum15}) and partial summation. \end{proof}

The next lemma bounds the number of $n\le x$ with two  large prime factors that are multiplicatively close.
\begin{lem}\label{MultiplicativelyCloseLargestandSecondLargestPrimeDivisors} For all large $x$, the number of $n\le x$ divisible by a pair of primes $q_1, q_2$ with
\[ x^{1/10\log_3 x} < q_1 \leq x\quad\text{and}\quad q_1 x^{-1/(\log_3 x)^2} \leq q_2 \leq q_1 \]
is $O(x/\log_3 x)$.
\end{lem}

\begin{proof} The number of such $n$ is at most  $x\sum_{x^{1/10\log_3 x}< q_1 \leq x} \frac1{q_1} \sum_{ q_1 x^{-1/(\log_3 x)^2} \leq q_2 \leq q_1} \frac1{q_2}$.
By Mertens' theorem, the inner sum is 
\begin{align*} \ll \log\left(\frac{\log{(q_1)}}{\log{(q_1 x^{-1/(\log_3 x)^2})}}\right) + \frac{1}{\log{(q_1 x^{-1/(\log_3 x)^2}})} \ll \frac{\log{x}}{(\log{q_1}) (\log_3{x})^2},
\end{align*}
leading to an upper bound for our count of $n$ of
\begin{equation*}  \ll \frac{x\log x}{(\log_3 x)^2} \sum_{x^{1/10\log_3 x} < q_1 \leq x} \frac1{q_1 \log q_1}  \ll \frac{x\log x}{(\log_3 x)^2} \cdot \frac{\log_3 x}{\log x} = \frac x{\log_3 x}. \end{equation*}
Here the final sum has been estimated by the prime number theorem and partial summation. \end{proof}

We conclude this section by quoting the main result  of \cite{wirsing59}.
\begin{lem}[Wirsing]\label{Wirsing}
    There exists an absolute constant $\lambda_0>0$ such that
    $$\#\left\{ m \leq x : \frac{\sigma(m)}m = \alpha \right\} \leq \exp\left(\lambda_0 \frac{\log x}{\log \log x} \right)$$
 for all $x \geq 3$ and all real numbers $\alpha$.
\end{lem}

\section{Proof of Theorem \ref{thm:main}}
As discussed in the introduction, it is enough to establish the following proposition. From now on, $y:=(\log\log{x})^{0.9}$.

\begin{prop}\label{prop:main} Fix $k\ge 4$. Almost always, $s(n)$ is not divisible by $p^k$ for any $p^k > y$. \end{prop}

We split the proof of Proposition \ref{prop:main} into two parts, according to the size of $p$.

\subsection{\dots when $y < p^k \le x^{1/2\log_3{x}}$} The following is a weakened form of Lemma 2.8 from \cite{pollack14}.

\begin{lem}\label{lem:pollack} For all large $x$, there is a set $\mathcal{E}(x)$ having size $o(x)$, as $x\to\infty$, such that the following holds. For all $d\le x^{1/2\log_3 x}$, the number of $n\le x$ not belonging to $\mathcal{E}(x)$ for which $d\mid s(n)$ is $O(x/d^{0.9})$.
\end{lem}

Summing the bound of Lemma \ref{lem:pollack} over $d=p^k$ with $y < p^k \le x^{1/2\log_3{x}}$ gives $o(x)$. It follows that almost always, $s(n)$ is not divisible by $p^k$ for any $p^k \in (y,x^{1/2\log_3 x}]$.

\subsection{\dots when $p^k > x^{1/2\log_3{x}}$}


The treatment of this range of $p$ is based on the following result, which may be of independent interest.

\begin{thm}\label{thm:divthm} For all large $x$, there is a set $\mathcal{E}(x)$ having size $o(x)$, as $x\to\infty$, such that the following holds. The number of $n\le x$ not belonging to $\mathcal{E}(x)$ for which $d\mid s(n)$ is
\[ \ll \frac{x}{d^{1/4}\log x} \]
uniformly for positive integers $d > x^{1/2\log_3 x}$ satisfying $P^{-}(d) > \log_2 x$.\end{thm}

The crucial advantage of Theorem \ref{thm:divthm} over Lemma \ref{lem:pollack} is the lack of any restriction on the size of $d$. Since $k\ge 4$, when we sum the bound of Theorem \ref{thm:divthm} over $d=p^k$ with $x^{1/2\log_3 x} < p^k < x^2$, the result is $O(x \log_2 x/\log{x})$, which is $o(x)$. So the proof of Theorem \ref{thm:main} will be completed once Theorem \ref{thm:divthm} is established. 

Turning to the proof of Theorem \ref{thm:divthm}, let $\mathcal{E}(x)$ denote the collection of $n \leq x$ for which at least one of the following fails:
\begin{enumerate}
    \item[(1)] $n> x/\log x$,
    \item[(2)] the largest squarefull divisor of $n$ is at most $\log_2 x$, 
    \item[(3)] $P^+(n) > x^{1/10\log_3 x}$,
    \item[(4)] $P^+(n)^2 \nmid n$,
    \item[(5)] $P^+(\gcd(n, \sigma(n))) \le \log_2 x$,
    \item[(6)] $P^+(n)>P_2^+(n)x^{1/(\log_3 x)^2}$, where $P_2^{+}(n):=P^{+}(n/P^{+}(n))$ is the second-largest prime factor of $n$.
\end{enumerate}

Let us show that only $o(x)$ integers $n\le x$ fail one of (1)--(6). This is obvious for (1). The count of $n\le x$ failing (2) is $\ll x \sum_{r > \log_2 x,~r\text{ squarefull}} 1/r \ll x/\sqrt{\log_2 x}$, and thus is $o(x)$. That the count of $n\le x$ failing (3) is $o(x)$ follows from standard bounds on the counting function of smooth (friable) numbers (e.g., Theorem 5.1 on p.\ 512 of \cite{tenenbaum15}), or Brun's sieve. The set of $n \leq x$ passing (3) but failing (4) has cardinality $\ll x \sum_{r>x^{1/10\log_3 x}}1/r^2 = o(x)$. Condition (5) is handled by Lemma \ref{gcd(n,sigma(n)) log log x smooth}. That the count of $n\le x$ satisfying (1)--(5) and failing (6) is $o(x)$ follows from Lemma \ref{MultiplicativelyCloseLargestandSecondLargestPrimeDivisors}.

Let $d$ be as in Theorem \ref{thm:divthm}. We separate the count of $n\notin \mathcal{E}(x)$ for which $d\mid s(n)$ according to whether $P^+(n) < d^{1/4} (\log x)^2$ or $P^+(n) \geq d^{1/4} (\log x)^2$. 

We first consider $n\notin \mathcal{E}(x)$ with $P^{+}(n) \ge d^{1/4}(\log{x})^2$.
Write $n=mP$, where $P:=P^{+}(n)$. Then $\gcd(m,P)=1$, and 
\[ x/m \ge d^{1/4} (\log{x})^2. \]
We can rewrite the condition $d\mid s(n)$ as $$ Ps(m) \equiv -\sigma(m) \pmod d.$$ For this congruence to have solutions, we must have $\gcd(s(m)\sigma(m), d)=1$. Indeed, if there exists a prime $q$ dividing both $\sigma(m)$ and $d$, then from $q\mid d$, we have $q> \log_2 x$, whereas since $d \mid s(n)$, we also have $q\mid s(n)$. But then the divisibility $q \mid \sigma(m)\mid \sigma(n)$ leads to $q \mid \gcd(n, \sigma(n))$, contradicting condition (5) above.  Since any common prime divisor of $s(m)$ and $d$ would, by the congruence, have to divide $\sigma(m)$ as well, we must indeed have $\gcd(s(m)\sigma(m), d)=1$. 

As such, the above congruence condition on $P$ places it in a unique coprime residue class modulo $d$. Hence, given $m$, the number of possible $P$ (and hence possible $n=mP$) is 
$$\ll \frac x{md} + 1 \ll \frac x{md} + \frac x{md^{1/4} (\log x)^2},$$
which when summed over $m \leq x$ is $\ll x/{d^{1/4} \log x}$, consistent with Theorem \ref{thm:divthm}. (We use here the lower bound on $d$.)

It remains to count $n\le x$, $n\notin \mathcal{E}(x)$ where $d\mid s(n)$ and $P^{+}(n) < d^{1/4}(\log x)^2$. For this case, we fix a constant
\[ \lambda> 2\lambda_0, \]
where $\lambda_0$ is the constant appearing in Wirsing's bound (Lemma \ref{Wirsing}). We will assume that $d\le x^{3/2}$, since $s(n) \le \sigma(n) < x^{3/2}$ for all $n\le x$, once $x$ is sufficiently large (e.g., as a consequence of the bound $\sigma(n) \ll n \log\log{(3n)}$; see Theorem 323 in \cite{Hw08}).

We write $n=AB$, where $A$ is the least unitary squarefree divisor of $n/P^+(n)$ exceeding $d^{1/4}\exp\left(\frac{\lambda}2 \frac{\log x}{\log_2 x}\right)$. Such a divisor exists as $n>x/\log x$ has maximal squarefull divisor at most $\log_2 x$, whereupon its largest unitary squarefree divisor coprime to $P^+(n)$ must be no less than 
$$\frac 1{d^{1/4} (\log x)^2 \log_2 x} \cdot \frac x{\log x} > d^{1/4} \exp\left(\frac{\lambda}2 \frac{\log x}{\log_2 x}\right).$$
(We assume throughout this argument that $x$ is sufficiently large.) Then \begin{equation}\label{eq:Bupper} B \leq \frac xA \leq \frac x{d^{1/4}}\exp\left(-\frac{\lambda}2 \frac{\log x}{\log_2 x}\right).\end{equation} Furthermore, 
\[ P^{+}(A) \le P_2^+(n)< P^+(n)x^{-1/(\log_3 x)^2} < d^{1/4} (\log x)^2 x^{-1/(\log_3 x)^2} < d^{1/4}x^{-\lambda/\log_2 x}. \]  Since $A/P^{+}(A)$ is a unitary squarefree divisor of $n/P^{+}(n)$, to avoid contradicting the choice of $A$, we must have $A \leq d^{1/2}\exp\left(-\frac{\lambda}2 \frac{\log x}{\log_2 x}\right)$. Then $\sigma(A) \ll A \log\log{A} \ll A\log_2 x$, so that (for large $x$) $\sigma(A) < d^{1/2}$.  

For each $B$ as above, we bound the number of corresponding $A$. First of all, since $\gcd(A, B)=1$, the divisibility $d\mid s(n)$ translates to the congruence $\sigma(A)\sigma(B) \equiv AB \pmod d$. Now, we claim that $\gcd(A\sigma(B), d)=1$: indeed, for any prime $q$ dividing both $A$ and $d$, we must have, on one hand, $q \geq P^-(d)>\log_2 x$, while on the other, $q\mid d \mid s(n)$ and $q \mid A \mid n$ imply $q\mid \gcd(n, \sigma(n))$. This contradicts (5). It follows by an analogous argument that $\gcd(\sigma(B), d)=1$, thus proving our claim. Consequently, the above congruence may be rewritten as 
$$\frac{\sigma(A)}A \equiv \frac B{\sigma(B)} \pmod d.$$
Now for some $B$, consider any pair of squarefree integers $A_1$ and $A_2$ satisfying the above congruence along with the conditions $\sigma(A_1), \sigma(A_2) < d^{1/2}$. Then $\sigma(A_1)/A_1 \equiv \sigma(A_2)/A_2 \pmod d$, leading to $\sigma(A_1)A_2 \equiv A_1\sigma(A_2) \pmod d$. But also $$|\sigma(A_1)A_2 - A_1\sigma(A_2)| \leq \max\{\sigma(A_1)A_2, A_1\sigma(A_2)\} < d,$$
thereby forcing $\sigma(A_1)/A_1 = \sigma(A_2)/A_2$. This shows that for each $B$, all corresponding $A$  have $\sigma(A)/A$ assume the same value, whereupon Lemma \ref{Wirsing} bounds the number of possible $A$  by $\exp\left(\lambda_0 \frac{\log x}{\log_2 x}\right)$. Keeping in mind the upper bound \eqref{eq:Bupper} on $B$, we deduce that the number of $n$ falling into this case is at most 
$$\frac x{d^{1/4}}\exp\left(-\frac{\lambda}2 \frac{\log x}{\log_2 x}\right) \cdot \exp\left(\lambda_0 \frac{\log x}{\log_2 x}\right) = \frac x{d^{1/4}} \exp\left(\left(\lambda_0 - \frac{\lambda}2\right)\frac{\log x}{\log_2 x}\right).$$
Since $\lambda > 2\lambda_0$, this final quantity is $\ll x/d^{1/4}\log{x}$. This completes the proof of Theorem \ref{thm:divthm}, and so also that of Theorem \ref{thm:main}.

\begin{rmk}\label{rmk:EGPS} Erd\H{o}s, Granville, Pomerance, and Spiro have conjectured \cite[Conjecture 4]{EGPS90} that $s^{-1}(\mathcal{A})$ has density $0$ whenever $\mathcal{A}$ has density $0$. If this holds, then the conclusion of Proposition \ref{prop:main} follows for each $k\ge 2$: take $$\mathcal{A} =\{ n\text{ divisible by $p^k$ for some $p^k > \log_3(100n)$}\}.$$ Unfortunately, very little is known in the direction of the EGPS conjecture. The record result (still quite weak) seems to be that of \cite{PPT18}, where it is shown that $s^{-1}(\mathcal{A})$ has density $0$ whenever $\mathcal{A}$ has counting function bounded by $x^{1/2+o(1)}$, as $x\to\infty$.
\end{rmk}
\providecommand{\bysame}{\leavevmode\hbox to3em{\hrulefill}\thinspace}
\providecommand{\MR}{\relax\ifhmode\unskip\space\fi MR }
\providecommand{\MRhref}[2]{%
  \href{http://www.ams.org/mathscinet-getitem?mr=#1}{#2}
}
\providecommand{\href}[2]{#2}


\begin{thebibliography}{EGPS90}

\bibitem[AE44]{AE44}
L.~Alaoglu and P.~Erd\H{o}s, \emph{A conjecture in elementary number theory},
  Bull. Amer. Math. Soc. \textbf{50} (1944), 881--882.

\bibitem[BL05]{BL05}
W.D. Banks and F.~Luca, \emph{Roughly squarefree values of the {E}uler and
  {C}armichael functions}, Acta Arith. \textbf{120} (2005), 211--230.

\bibitem[BP06]{BP06}
W.D. Banks and F.~Pappalardi, \emph{Values of the {E}uler function free of
  {$k$}th powers}, J. Number Theory \textbf{120} (2006), 326--348.

\bibitem[EGPS90]{EGPS90}
P.~Erd\H{o}s, A.~Granville, C.~Pomerance, and C.~Spiro, \emph{On the normal
  behavior of the iterates of some arithmetic functions}, Analytic number
  theory ({A}llerton {P}ark, {IL}, 1989), Progr. Math., vol.~85, Birkh\"{a}user
  Boston, Boston, MA, 1990, pp.~165--204.

\bibitem[ELP08]{ELP08}
P.~Erd\H{o}s, F.~Luca, and C.~Pomerance, \emph{On the proportion of numbers
  coprime to a given integer}, Anatomy of integers, CRM Proc. Lecture Notes,
  vol.~46, Amer. Math. Soc., Providence, RI, 2008, pp.~47--64.

\bibitem[HW08]{Hw08}
G.H. Hardy and E.M. Wright, \emph{An introduction to the theory of numbers},
  sixth ed., Oxford University Press, Oxford, 2008.

\bibitem[LP15]{LP15}
F.~Luca and C.~Pomerance, \emph{The range of the sum-of-proper-divisors
  function}, Acta Arith. \textbf{168} (2015), 187--199.

\bibitem[Pap03]{pappalardi03}
F.~Pappalardi, \emph{Square free values of the order function}, New York J.
  Math. \textbf{9} (2003), 331--344.

\bibitem[Pap05]{pappalardi05}
\bysame, \emph{A survey on {$k$}-freeness}, Number theory, Ramanujan Math. Soc.
  Lect. Notes Ser., vol.~1, Ramanujan Math. Soc., Mysore, 2005, pp.~71--88.

\bibitem[Pol14]{pollack14}
P.~Pollack, \emph{Some arithmetic properties of the sum of proper divisors and
  the sum of prime divisors}, Illinois J. Math. \textbf{58} (2014), 125--147.

\bibitem[Pom77]{pomerance77}
C.~Pomerance, \emph{On the distribution of amicable numbers}, J. Reine Angew.
  Math. \textbf{293(294)} (1977), 217--222.

\bibitem[PPT18]{PPT18}
P.~Pollack, C.~Pomerance, and L.~Thompson, \emph{Divisor-sum fibers},
  Mathematika \textbf{64} (2018), 330--342.

\bibitem[PSS03]{PSS03}
F.~Pappalardi, F.~Saidak, and I.E. Shparlinski, \emph{Square-free values of the
  {C}armichael function}, J. Number Theory \textbf{103} (2003), 122--131.

\bibitem[Ten15]{tenenbaum15}
G.~Tenenbaum, \emph{Introduction to analytic and probabilistic number theory},
  third ed., Graduate Studies in Mathematics, vol. 163, American Mathematical
  Society, Providence, RI, 2015.

\bibitem[Wir59]{wirsing59}
E.~Wirsing, \emph{Bemerkung zu der {A}rbeit \"{u}ber vollkommene {Z}ahlen},
  Math. Ann. \textbf{137} (1959), 316--318.

\end{thebibliography}
\end{document}